\documentclass[11pt]{article}
\usepackage{amsmath,amsthm,amssymb,hyperref,tikz,ifthen,wrapfig}
\usetikzlibrary{calc}

\usepackage[margin=1.1in]{geometry}

\hypersetup{
pdfauthor = {Boris Bukh and Alfredo Hubard},
pdftitle = {Space crossing numbers},
pdfsubject = {Mathematics},
pdfkeywords = {crossing number, quadrisecants, linking number, graph embeddings},
pdfstartview = {FitH},
pdfpagemode = {None}}

\title{Space crossing numbers}
\author{Boris Bukh\footnote{\texttt{B.Bukh@dpmms.cam.ac.uk}, University of Cambridge 
and Churchill College, United Kingdom.}
\and Alfredo Hubard\footnote{\texttt{hubard@cims.nyu.edu}. Courant Institute of Mathematical Sciences, 
New York University, United States.}}

\date{}

\newcommand*{\abs}[1]{\lvert #1\rvert}                         
\newcommand*{\norm}[1]{\lVert #1\rVert}                        
\newcommand*{\R}{\mathbb{R}}                                   
\newcommand*{\Z}{\mathbb{Z}}                                   
\newcommand*{\N}{\mathbb{N}}                                   
\newcommand*{\Sph}{\mathbb{S}}                                 
\newcommand*{\TSph}{T\mathbb{S}}                               
\newcommand*{\Thom}{\mathcal{T}}                               
\newcommand*{\E}{\mathbb{E}}                                   
\newcommand*{\F}{\mathcal{F}}                                  
\newcommand*{\veps}{\varepsilon}                               
\newcommand*{\defeq}{\stackrel{\text{\tiny def}}{=}}           
\newcommand*{\bd}{\mathbf{d}}                                  
\DeclareMathOperator{\crn}{cr}                                 
\DeclareMathOperator{\rcrn}{lin-cr}                            
\DeclareMathOperator{\lk}{lk}                                  
\DeclareMathOperator{\sign}{sgn}                               
\DeclareMathOperator{\conv}{conv}                              
\DeclareMathOperator{\convcone}{conv-cone}                     
\DeclareMathOperator{\dist}{dist}                              

\newtheorem{theorem}{Theorem}
\newtheorem{lemma}[theorem]{Lemma}
\newtheorem{proposition}[theorem]{Proposition}
\newtheorem{corollary}[theorem]{Corollary}

\newenvironment{remark}[1][Remark]
{\begin{trivlist}\item\textbf{#1:\ }} {\end{trivlist}}

\makeatletter
\newcommand*{\proofc}[1][\proofname]{\noindent\textit{#1.} \ignorespaces}

\begin{document}
\maketitle

\begin{abstract}
We define a variant of the crossing number for an embedding of a graph $G$ into $\R^3$,
and prove a lower bound on it which almost implies the classical crossing lemma.
We also give sharp bounds on the rectilinear space crossing numbers of pseudo-random graphs.
\end{abstract}

\section*{Introduction}
All the graphs in this paper are simple, i.e.\ they contain no loops or multiple edges. 
The \emph{crossing number} of a graph $G=(V,E)$ is the minimum number of crossings between
edges of $G$ among all the ways to draw $G$ in the plane. It is denoted $\crn(G)$. 
The edges in a drawing of $G$ need not be line segments, they are allowed to be 
arbitrary continuous curves. If one restricts to the straight-line drawings, then
one obtains the \emph{rectilinear crossing number} $\rcrn(G)$. It is clear that
$\crn(G)\leq \rcrn(G)$, and there are examples where $\crn(G)=4$, but $\rcrn(G)$ is
unbounded \cite{bienstock_dean}. The principal result about crossing numbers is the crossing
lemma of Ajtai--Chv\'atal--Newborn--Szemer\'edi and Leighton \cite{orig_crossing,orig_crossing_leighton} 
which states that
\begin{equation}\label{eq_crossing_lemma}
\crn(G)\geq c\frac{\abs{E}^3}{\abs{V}^2}\qquad \text{whenever }\abs{E}\geq C\abs{V}.
\end{equation}
The inequality is sharp apart from the values of $c$ and $C$ (see \cite{pach_toth_constant} for
the best known estimate on $c$). 
The most famous applications of the crossing lemma are short and elegant proofs 
by Sz\'ekely \cite{szekely_crossing} of Szemer\'edi--Trotter theorem on point-line 
incidences and of Spencer--Szemer\'edi--Trotter theorem on the unit distances. 
Another remarkable application is the bound on the number of halving lines by 
Dey\cite{dey_ksets}. In this paper we propose an extension of 
the crossing number to $\R^3$, in such a way that the corresponding ``space crossing 
lemma'' (Theorem~\ref{thm_space_crossing_lemma} below) implies 
\eqref{eq_crossing_lemma} (up to a logarithmic factor). 

A \emph{spatial drawing} of a graph $G$ is representation of vertices 
of $G$ by points in $\R^3$,  and edges of $G$ by continuous curves. A 
\emph{space crossing} consists of a quadruple of vertex-disjoint 
edges $(e_1,\dotsc,e_4)$ and a line $l$ that meets these four edges. 
The \emph{space crossing number} of $G$, denoted $\crn_4(G)$ is 
the least number of crossings in any spatial drawing of $G$. As 
in the planar case, the \emph{spatial rectilinear crossing number} 
$\rcrn_4(G)$ is obtained by restricting to straight-line spatial drawings.

For a graph $G$ pick a drawing of $G$ in the plane with the fewest 
crossings. By perturbing the drawing  slightly, we may assume that 
there are no points where three vertex-disjoint edges meet.
The drawing can be lifted to a drawing $G$ on a large sphere without 
changing any of the crossings. Since no line meets the sphere 
in more than two points, every space crossing in the resulting 
spatial drawing comes from a pair of crossings in the planar 
drawing. Thus,
\begin{equation}\label{eq_plane_vs_space}
\crn_4(G)\leq \binom{\crn(G)}{2}.
\end{equation}
Let us note that the space crossing number is not the usual crossing number in disguise,
for the inequality in the reverse direction does not hold: 
\begin{proposition}[Proof is on p.~\pageref{pf_prop_incomparability}]\label{prop_incomparability}
For every natural number $n$ there is a graph $G$
with $\crn_4(G)=0$ and $\crn(G)\geq n$. 
\end{proposition}
The principal result that justifies the introduction of the space crossing number 
is the following generalization of the crossing lemma.
\begin{theorem}[Proof is on p.~\pageref{pf_thm_space_crossing_lemma}]\label{thm_space_crossing_lemma}
Let $G=(V,E)$ be an arbitrary graph, then
\begin{equation*}
\crn_4(G)\geq \frac{\abs{E}^6}{4^{179}\abs{V}^4\log_2^2 \abs{V}},
\end{equation*}
whenever $\abs{E}\geq 4^{41}\abs{V}$.
\end{theorem}
Since \eqref{eq_crossing_lemma} is sharp, in the light of
the argument that led to \eqref{eq_plane_vs_space}  there
are graphs on the sphere for which the bound in Theorem~\ref{thm_space_crossing_lemma} 
is tight up  to the logarithmic factor. In the drawings of these graphs, 
the edges are of course not straight. It turns out that 
there are also straight-line spatial drawings for which 
Theorem~\ref{thm_space_crossing_lemma} is tight.
\begin{theorem}[Proof is on p.~\pageref{pf_thm_constr}]\label{thm_constr}
For all positive integers $m$ and $n$ satisfying $m\leq \binom{n}{2}$ there is
a graph $G$ with $n$ vertices and $m$ edges, and rectilinear space crossing number
at most $6720 m^6/n^4$.
\end{theorem}
The construction in the proof of Theorem~\ref{thm_constr} uses the 
idea of  stair-convexity introduced in \cite{bmn_weakeps}. We shall 
briefly review the necessary background before the proof of Theorem~\ref{thm_constr}.

Our final result is the lower bound on the space crossing number 
of (possibly sparse) pseudo-random graphs. 
\begin{theorem}[Proof is on p.~\pageref{pf_thm_expand}]\label{thm_expand}
There is an absolute constant $\veps>0$ such that the following holds.
Let $G=(V,E)$ be a graph such that whenever $V_1,V_2$ are any two subsets of
$V$ of size $\veps\abs{V}$, the number of edges between $V_1$ and $V_2$ 
is at least $N$. Then $\rcrn_4(G)\geq N^4$. 
\end{theorem}
The condition of the theorem holds for several models of random graphs,
as well as for $(n,d,\lambda)$-graphs (see for example \cite[Theorem 2.11]{ks_survey}).

\section*{Separation between crossing numbers and space crossing numbers}
To construct graphs with $\crn_4(G)=0$ and unbounded $\crn(G)$, we shall
use the lower bound on crossing numbers due to Riskin. Recall that a $3$-connected
planar graph has  a unique planar drawing \cite[Theorem~4.3.1]{diestel}.
\begin{lemma}[Theorem~4 in \cite{riskin_edge}]
Suppose $e$ is an edge in a graph $G$ such that $H=G\setminus e$ is
a $3$-regular $3$-connected planar graph. Then there is a drawing
of $G$ in the plane with $\crn(G)$ crossings that is obtained from the unique 
planar drawing of $H$ by adding the edge $e$.
\end{lemma}
\proofc[Proof of Proposition~\ref{prop_incomparability}]\label{pf_prop_incomparability}
Let $H$ be the truncated $n$-by-$n$ hexagonal grid
drawn 
\begin{wrapfigure}{r}[1.5cm]{0cm}
\vspace{-2ex}
\begin{tikzpicture}[scale=0.5]
\def\gridsize{2}
\def\gridsizep1{3} 
\def\gridsizem2{0} 
\foreach \y in {-\gridsize,...,\gridsize}
 \foreach \x in {-\gridsize,...,\gridsize}
   {
    \ifthenelse{\NOT\(\equal{\x}{-\gridsize} \AND \equal{\y}{\gridsize}\) 
           \AND \NOT\(\equal{\x}{\gridsize} \AND \equal{\y}{-\gridsize}\)  }
    {
     \draw ($\x*(1.5,{0.5*sqrt(3)})+\y*(-1.5,{0.5*sqrt(3)})$)
       +(0:1)  -- +(60:1) -- +(120:1) -- +(180:1) -- +(240:1) -- +(300:1) -- cycle;
     \filldraw ($\x*(1.5,{0.5*sqrt(3)})+\y*(-1.5,{0.5*sqrt(3)})$)
       +(0:1) circle (2.3pt)
       +(60:1) circle (2.3pt)
       +(120:1) circle (2.3pt)
       +(180:1) circle (2.3pt)
       +(240:1) circle (2.3pt) 
       +(300:1) circle (2.3pt);
    }
    {
     \ifthenelse{\equal{\x}{-\gridsize} \AND \equal{\y}{\gridsize}}
     {
      \draw ($\x*(1.5,{0.5*sqrt(3)})+\y*(-1.5,{0.5*sqrt(3)})$)
         +(-60:1) -- +(60:1);
     }
     {
      \draw ($\x*(1.5,{0.5*sqrt(3)})+\y*(-1.5,{0.5*sqrt(3)})$)
         +(120:1) -- +(240:1);
     }
    } 
   }
\foreach \y in {-\gridsizem2,...,\gridsize}
{
 \draw (${\gridsize+1}*(1.5,{0.5*sqrt(3)})+\y*(-1.5,{0.5*sqrt(3)})$) 
    +(180:1) -- +(300:1);
}
\foreach \y in {-\gridsize,...,\gridsizem2}
{
 \draw (${-\gridsize-1}*(1.5,{0.5*sqrt(3)})+\y*(-1.5,{0.5*sqrt(3)})$) 
    +(0:1) -- +(120:1);
}
\foreach \x in {-\gridsizem2,...,\gridsize}
{
 \draw ($\x*(1.5,{0.5*sqrt(3)})+{\gridsize+1}*(-1.5,{0.5*sqrt(3)})$) 
    +(0:1) -- +(-120:1);
}
\foreach \x in {-\gridsize,...,\gridsizem2}
{
 \draw ($\x*(1.5,{0.5*sqrt(3)})+{-\gridsize-1}*(-1.5,{0.5*sqrt(3)})$) 
    +(60:1) -- +(180:1);
}
\fill ($-1*(1.5,{0.5*sqrt(3)})+0*(-1.5,{0.5*sqrt(3)})$) ++(0:-1) circle (5pt) +(0:0.45) node {$u$}; 
\fill ($1*(1.5,{0.5*sqrt(3)})+0*(-1.5,{0.5*sqrt(3)})$) ++(0:1) circle (5pt) +(-60:0.45) node {$v$}; 
\end{tikzpicture}\vspace{-2ex}
\end{wrapfigure}
as in the picture on 
the right. The graph $H$ is clearly $3$-connected $3$-regular planar graph. 
Pick two vertices $u,v\in H$ that are separated from
one another by at least $n/4$ faces (the outer region is 
also a face). Then by the preceding lemma
the graph $G=H\cup \{uv\}$ has crossing number at least $n/4$.
On the other hand, there is a spatial drawing of $G$ without any
spatial crossings: Let $H$ be drawn on the surface of the sphere without
crossings, and represent the edge $uv$ by a straight-line segment.
Since every line meets the sphere in at most two vertex-disjoint edges,
there are indeed no space crossings. \qed

\section*{Lower bounds on the space crossing number}\label{pf_thm_space_crossing_lemma}
The naive strategy to prove Theorem~\ref{thm_space_crossing_lemma} is to show
that a graph without space crossing can have only $O(\abs{V})$ edges, derive from
this a lower bound on the space crossing number of the form $c_1 \abs{E}-c_2 \abs{V}$,
and then use random sampling to ``boost'' this to a stronger bound on $\crn_4$.
Whereas, it is true that a space-crossing-free graph has only $O(\abs{V})$
edges (it follows from \cite[Corollary~3.5]{zivaljevic_k6} that a graph with a 
$K_{6,6}$-minor has a space crossing),
this approach yields only $\crn_4(G)\geq c \abs{E}^7/\abs{V}^6$. The reason is that 
to get  $\crn_4\geq c \abs{E}^6/\abs{V}^4$ one needs to boost a stronger inequality 
$\crn_4(G)\geq c_1 \abs{E}^2 - c_2 \abs{V}^2$. 
To obtain such an inequality we shall break the graph $G$ into many 
small pieces, so that for each pair of pieces
there is a space crossing that involves two edges from each piece.
For that we need several known results, which we now state.

Recall that a subdivision of a graph $G$ is a graph obtained from $G$ by subdividing
each edge of $G$ into paths \cite[p.~20]{diestel}.
\begin{lemma}[\cite{kostochka_pyber}]\label{lem_kostochka_pyber}
Let $\veps>0$ be arbitrary. Then every  
graph $G=(V,E)$ with $4^{t^2}\abs{V}^{1+\veps}$ edges contains a subdivision
of $K_t$ on at most $7t^2\log t/\veps$ vertices.
\end{lemma}
\begin{corollary}\label{cor_manydisjoint}
Let $C\geq 3$. Suppose $G=(V,E)$ is a graph with at least $\abs{E}\geq C 4^{t^2} \abs{V}$ 
edges. Then $G$ contains at least $\abs{E}/(16 t^2 \log t\log_{C/2} \abs{V})$ edge-disjoint subdivisions of $K_t$.
\end{corollary}
\begin{proof}
Define a nested sequence of graphs $G=G_0\supset G_1\supset G_2\supset \dotsb\supset G_s$ on the vertex 
set $V$ as follows. As long as $E(G_i)\geq (C/2) 4^{t^2} \abs{V}$ it follows by Lemma~\ref{lem_kostochka_pyber}
with $\veps=1/\log_{C/2} \abs{V}$ that $G_i$ contains a subgraph $H_i$, which is a subdivision
of $K_t$ with $\abs{V(H_i)}\leq 7t^2\log t \log_{C/2} \abs{V}$. Let $G_{i+1}$ be the result
of removing the edges of $H_i$ from $G_i$. The sequence terminates 
once the number of edges in the graph falls below $(C/2) 4^{t^2} \abs{V}$.
As 
\begin{equation*}
\abs{E(G_i)}-\abs{E(G_{i+1})}=\abs{E(H_i)}
\leq \abs{E(K_t)}+\abs{V(H_i)}\leq 8t^2 \log t \log_{C/2} \abs{V},
\end{equation*}
the number of terms in the sequence is at least $(\abs{E}/2)/(8t^2 \log t \log_{C/2} \abs{V})$.
Since the graphs $\{H_i\}$ are edge-disjoint subgraphs of $G$, the corollary follows.
\end{proof}
The next is a version of \cite[Theorem~3]{bs_judicious}.
\begin{lemma}\label{lem_partitiongraph}
The vertex set of every graph $G=(V,E)$ can be partitioned into two 
classes $V=V_1\cup V_2$ so that the number of edges in each of the 
induced subgraphs $G_i=G|_{V_i}$ is at least
$\abs{E}/4-\sqrt{\abs{V}\abs{E}}$.
\end{lemma} 
\begin{proof}
For each vertex $v$ place $v$ into $V_1$ or $V_2$ with equal probability independently
of the other vertices. Let $X_i=\abs{E(G_i)}$. Then $\E[X_i]=\tfrac{1}{4}\abs{E}$ and
\begin{align*}
\E[X_i^2]&=\sum_{e_1,e_2\in E} \Pr[e_1\in E_i\wedge e_2\in E_i]\\&=
\tfrac{1}{4}\abs{E}+\tfrac{1}{8}\abs{\{(e_1,e_2)\in E^2: \abs{e_1\cap e_2}=1\}}+
\tfrac{1}{16}\abs{\{(e_1,e_2)\in E^2 : e_1\cap e_2=\emptyset\}}\\
&\leq \abs{E}^2/16+\tfrac{1}{4}\sum_{v\in V} \deg(v)^2\leq \abs{E}^2/16+\tfrac{1}{4}\abs{V}\abs{E}.
\end{align*}
Hence $\E\bigl[(X_i-\abs{E}/4)^2\bigr]\leq \tfrac{1}{4}\abs{V}\abs{E}$, implying
$\Pr\bigl[\abs{X_i-\abs{E}/4}>\sqrt{\abs{V}\abs{E}}\bigr]<1/4$. Therefore, the conclusion
of the lemma holds for the random partition with probability at least $1/2$.
\end{proof}

To find lines through four edges, we use two results from knot theory. 
We first recall the standard definitions. Two continuous injective 
maps $f_1,f_2\colon \Sph^1\to \R^2$ whose images are disjoint define 
a (two-component) \emph{link} in $\R^3$. The 
sets $C_1=f_1(\Sph^1)$ and $C_2=f_2(\Sph^1)$ are a pair of continuous closed 
curves (knots) in $\R^3$. The linking number $\lk(C_1,C_2)$ of the two curves 
is the degree\footnote{Implicit in the definition of 
the degree is the group of the coefficients for the homology.
We use $\Z$ coefficients throughout the paper.} of the Gauss map 
\begin{equation}\label{eq_gaussmap}
\begin{aligned}
g&\colon \Sph^1\times \Sph^1\to \Sph^2\\ 
g&\colon (x,y)\mapsto \frac{f_1(x)-f_2(y)}{\norm{f_1(x)-f_2(y)}}.
\end{aligned}
\end{equation}
The linking number is an invariant of the knots, and if the 
functions $f_1,f_2$ are sufficiently nice, then it can also 
be defined by counting the number of signed crossings 
between $C_1$ and $C_2$ in a projection to a generic plane.
\begin{lemma}[Theorem 1 in \cite{conway_gordon} and independently in \cite{sachs}]\label{lem_conway_gordon}
In every spatial drawing of $K_6$ there is a pair of vertex-disjoint triangles
whose linking number is odd.
\end{lemma}
\begin{lemma}\label{lem_linking}
If $C_1,C_2,C_3,C_4\subset \R^3$ are four disjoint continuous closed curves, and
$\lk(C_1,C_2)$ and $\lk(C_3,C_4)$ are non-zero, then there is at least
one line that intersects all the four curves.
\end{lemma}
This lemma is similar to Corollary~1 of Theorem~2 in \cite{viro}. That corollary
asserts that if the four curves are in addition smooth, and satisfy an
appropriate general position requirement, then the number of lines through
all four of them is at least $\abs{\lk(C_1,C_2)\lk(C_3,C_4)}$. 
It is possible to derive Lemma~\ref{lem_linking} from the result in \cite{viro},
by a limiting argument. For completeness we include a short 
proof of Lemma~\ref{lem_linking}, which uses a different idea.

\begin{proof}[Proof of Lemma~\ref{lem_linking}]
Let $\TSph^2$ be the tangent bundle to $\Sph^2$. An element 
$(p,v)\in \TSph^2$ consists of a point $p\in \Sph^2$ and a tangent
vector $v$ to $p$. We shall think of $(p,v)\in\TSph^2$ as a directed
line in $\R^3$ in direction $p$ which intersects the 
hyperplane $\{x : \langle x,p\rangle=0\}$ in the point $v$.

For each $i=1,\dotsc,4$, let $f_i\colon \Sph^1\to\R^3$ be a continuous 
injective map such that $f_i(\Sph^1)=C_i$. Consider the pair $f_1,f_2$,
and for $x,y\in\Sph^1$ let $h_{12}(x,y)\in \TSph^2$ be the directed
line that goes from $f_2(y)$ to $f_1(x)$. The result of composition
of $h_{12}\colon \Sph^1\times \Sph^1\to \TSph^2$ with the projection
map $\pi\colon \TSph^2\to \Sph^2$ is the Gauss map $g_{12}=\pi\circ h_{12}$
as defined in \eqref{eq_gaussmap}. By the assumption
the degree of $g_{12}$ is non-zero. Since $\Sph^2$ is a
deformation retract of $\TSph^2$, the projection map $\pi$
induces isomorphism between the homology groups
of $\TSph^2$ and $\Sph^2$, and hence the degree of $h_{12}$ is non-zero.

Let $\Thom=\Thom(\TSph^2)$ be the Thom space 
of $\TSph^2$. It is a bundle over $\Sph^2$ obtained
from $\TSph^2$ replacing each fiber by its one-point compactification,
and identifying all the new points into a single point 
(see page~367 of \cite{bredon} for the motivation and properties).
Let $\sigma\colon \TSph^2\to \Thom$ be the inclusion map. 
Let $A_{12}\defeq(\sigma\circ h_{12})(\Sph^1\times \Sph^1)$ be
the image of $h_{12}$ in $\Thom$. 

In the same way as we used $f_1$ and $f_2$ to define $h_{12}$ and $A_{12}$,
we define $h_{34}$ and $A_{34}$ using $f_3$ and $f_4$. 
We shall exhibit two homology classes
$\alpha_{12}\in H_2(A_{12},\Z)$ and $\alpha_{34}\in H_2(A_{34},\Z)$
whose intersection product in $\Thom$ is non-zero. It will then
follow by Theorem~VI.11.10 from \cite{bredon} that $A_{12}\cap A_{34}\neq\emptyset$.

Since $\pi$ induces an isomorphism between 
$H_2(\TSph^2,\Z)$ and $H_2(\Sph^2,\Z)$, the definition 
of the linking number implies that the pushforward of 
the homology class $[\Sph^1\times \Sph^1]\in H_2(\Sph^1\times \Sph^1,\Z)$ 
by $h_{12}$ is the homology class $\lk(C_1,C_2)[\Sph^2]\in H_2(\TSph^2,\Z)$.
Let $D\colon H^k(M,\Z)\to H_{\dim M-k}(M,\Z)$ be a Poincare duality
on orientable manifold $M$. The homology class $\sigma_*([\Sph^2])$ is the 
$D^{-1}(\tau)$, where $\tau\in H^2(\Thom)$ is the Thom class of $\Thom$. 
The homology class 
$\alpha_{12}\defeq (\sigma\circ h_{12})_*([\Sph^1\times \Sph^1])$
is supported on $A_{12}$, and similarly defined class 
$\alpha_{34}$ is supported on $A_{34}$. The intersection
product of $\alpha_{12}$ and $\alpha_{34}$ is then 
$\lk(C_1,C_2)\lk(C_3,C_4)D^{-1}(\tau^2)$. By the calculation
on page~382 of \cite{bredon} $(\tau^2)\cap [\TSph^2]=i_*(\chi \cap [\Sph^2])$, 
where $\chi$ is the Euler class of the bundle $\TSph^2\to \Sph^2$ and $i\colon \Sph^2\to \TSph^2$
is the zero section. Thus $\tau^2$ is non-zero, and hence the intersection product of $\alpha_{12}$ and $\alpha_{34}$
is non-zero as well, as claimed.
\end{proof}

The following lemma is analogous to the inequality $\crn(G)\geq
\abs{E}-3\abs{V}+6$ that is used in the proof of the usual crossing lemma.
\begin{lemma}\label{lem_boost}
Let $G=(V,E)$ be a graph with at least $\abs{E}\geq 4^{39}\abs{V}$ edges.
Then $\crn_4(G)\geq \abs{E}^2/2^{28}\log_2^2 \abs{V}$. 
\end{lemma}
\begin{proof}
With foresight set $J=\bigl\lceil \abs{E}/2^{14}\log_2\abs{V}\bigr\rceil$.
By Lemma~\ref{lem_partitiongraph} the graph splits into two vertex-disjoint 
graphs $G_1,G_2$ that have at least $\abs{E}/4-\sqrt{\abs{V}\abs{E}}\geq 
\abs{E}(1/4-1/4^{19})\geq \abs{E}/8$ 
edges each. By Corollary~\ref{cor_manydisjoint} each of $G_i$ contains
a family of $\abs{E}/(8\cdot 16\cdot 6^2 \log 6 \log_2 \abs{V})\geq J$ edge-disjoint subdivisions 
of $K_6$. Thus, by 
Lemma~\ref{lem_conway_gordon} we obtain a family of $J$ pairs of cycles
$(C_{i,j},C_{i,j}')_{j=1}^J$ in $G_i$, such that $C_{i,j}$ and $C_{i,j}'$ are
vertex-disjoint, all the cycles are edge-disjoint, and $\lk(C_{i,j},C_{i,j}')\geq 1$. 
By lemma~\ref{lem_linking} for every $1\leq j_1,j_2\leq J$ there is a line
that intersects $C_{1,j_1},C_{1,j_1}',C_{2,j_2},C_{2,j_2}'$. Furthermore, the four
cycles are vertex disjoint. As all the cycles are edge-disjoint, the
$J^2$ space crossings obtained in this manner are distinct. 
\end{proof}
\begin{corollary}
If $G$ is any graph, and $B\geq \abs{V}$ 
then $\crn_4(G)\geq \frac{\abs{E}^2-4^{80}\abs{V}^2}{2^{28}\log_2^2 B}$.
\end{corollary}

\begin{proof}[Proof of Theorem~\ref{thm_space_crossing_lemma}]
Given a graph $G=(V,E)$ with $\abs{E}\geq 4^{41}\abs{V}$ edges,
let $p=4^{41}\abs{V}/\abs{E}$. Let $V'\subset V$ be obtained 
by choosing each element of $V$ independently with probability $p$.
Let the $G'=(V',E')$ be the induced subgraph $G$ on $V'$. By the preceding 
corollary with $B=\abs{V}$ we have
\begin{equation*}
\crn_4(G')\geq \frac{\abs{E'}^2-4^{80}\abs{V'}^2}{2^{26}\log_2^2 \abs{V}}.
\end{equation*}
We shall estimate the expectation of both sides.
On one hand, $\E[\crn_4(G')]\leq p^8\crn_4(G)$ since a space crossing
in a fixed drawing survives with probability $p^8$. On the other hand,
$\E[\abs{E'}^2]\geq p^4\abs{E}^2$ 
as every pair of edges survives with probability at least $p^4$
(the probability is higher if the two edges overlap). 
Furthermore, $\E[\abs{V'}^2]=p^2\abs{V}^2+(p-p^2)\abs{V}\leq 4p^2\abs{V}^2$ by the
choice of $p$. Hence,
\begin{equation*}
p^8\crn_4(G)\geq \frac{p^4\abs{E}^2-4^{81}p^2\abs{V}^2}{2^{28}\log_2^2 \abs{V}},
\end{equation*}
and
\begin{equation*}
\crn_4(G)\geq \frac{4^{81}\abs{V}^2}{2^{28}(4^{41}\abs{V}/\abs{E})^6\log_2^2\abs{V}}
=\frac{\abs{E}^6}{4^{179}\abs{V}^4\log_2^2\abs{V}}\qedhere
\end{equation*}
\end{proof}
\begin{remark} By using Lemma~\ref{lem_boost} instead of its corollary,
and invoking large deviation inequalities, the above can be improved to 
$\crn_4(G)\geq c\frac{\abs{E}^6}{\abs{V}^4\log_2^2(\abs{V}^2/\abs{E})}$.
As the logarithmic factors are almost certainly superfluous, we chose
the more transparent argument instead.
\end{remark}

\section*{Rectilinear space crossing numbers of pseudo-random  graphs}\label{pf_thm_expand}
To prove Theorem~\ref{thm_expand} we shall need the same-type lemma
for semi-algebraic relations. It is inspired by the same-type lemma of
B\'ar\'any and Valtr \cite[Theorem~2]{barany_valtr}, and by the Szemer\'edi-type 
result from \cite{fglnp}. In the final version of \cite{fglnp}
(to appear in J.~Reine Angew.~Math), a more general result is proved 
independently. Our proof technique is borrowed from the previous results, with 
only minor pretense at novelty.

For a real number $x$ its sign $\sign x$ is $-1,0,+1$ according to whether
$x$ is negative, zero, or positive, respectively.
A semi-algebraic relation on $k$-tuples of vectors $x_1,\dotsc,x_k$
is an arbitrary logical formula (in the language of ordered fields) 
of the form 
\begin{equation*}
Q_1 t_1\in\R\, Q_2 t_2\in \R\, \dotsb Q_l t_l\in \R\ (I_1 \wedge \dotsb \wedge I_m)
\end{equation*}
where each of $Q_1,\dotsc,Q_l$ is either $\exists$ or $\forall$
and each of $I_1,\dotsc,I_m$ is of the form 
$\sign f(x_1,\dotsc,x_k,t_1,\dotsc,t_l)=s\in\{-1,0,+1\}$, where $f$ is a polynomial. 

\begin{lemma}[Proof is in on p.~\pageref{pf_lem_sametype}]\label{lem_sametype}
If $R$ is a semi-algebraic relation in $k$ variables, then
there is a constant $\veps=\veps(R)>0$ such that the following holds.
For every collection of $k$ finite sets $\F_1,\dotsc,\F_k$, there are
subsets $\F_i'\subset\F_i$ such that:
\begin{enumerate}
\item $\F_i'$ are large: $\abs{\F_i'}\geq \veps \abs{\F_i}$,
\item $R$ is constant on $\F_1'\times\dotsb\times \F_k'$: either for all 
$(x_1,\dotsc,x_k)\in\F_1'\times\dotsb\times \F_k'$ the relation $R(x_1,\dotsc,x_k)$
holds, or for all $(x_1,\dotsc,x_k)\in \F_1'\times\dotsb\times \F_k'$ the relation
$R(x_1,\dotsc,x_k)$ does not hold.
\end{enumerate}
\end{lemma}

\begin{proof}[Proof of Theorem~\ref{thm_expand}]
Let the graph $G$ with a rectilinear spatial drawing be given. 
Let $R$ be the relation on $8$-tuples $x_1,\dotsc,x_8\in\R^3$ given by
``the straight-line segments $x_1x_2$, $x_3x_4$, $x_5x_6$, $x_7x_8$
form a space crossing''. The relation is semi-algebraic. Indeed, 
it is given by
\begin{align*}
&R(x_1,\dotsc,x_8)=\exists t_1,t_2,t_3,t_4,\lambda_1,\lambda_2,\lambda_3,\lambda_4\in\R,\ 
\exists y,v\in\R^3\\
&\qquad (0<t_1,t_2,t_3,t_4<1)\wedge \\
&\qquad(t_1x_1+(1-t_1)x_2=y+\lambda_1 v) \wedge 
(t_2x_3+(1-t_2)x_4=y+\lambda_2 v) \wedge\\
&\qquad(t_3x_5+(1-t_3)x_6=y+\lambda_3 v) \wedge
(t_4x_7+(1-t_4)x_8=y+\lambda_4 v). 
\end{align*}

By the preceding lemma applied $8!\binom{12}{8}$ times 
there are $12$ subsets $V_1,\dotsc,V_{12}$ of $V(G)$
such that $\abs{V_i}\geq \veps\abs{V}$ for $i=1,\dotsc,12$ and
$R$ is constant on all the product sets of the form 
$V_{\sigma(1)}\times\dotsb\times V_{\sigma(8)}$ for any injective map $\sigma\colon [8]\to [12]$. 
Pick any twelve points $x_1\in V_1,\dotsc,x_{12}\in V_{12}$.
Since graph $K_{12}$ contains $K_{6,6}$, which has a 
positive space crossing by \cite[Corollary~3.5]{zivaljevic_k6}, there is a
map $\sigma\colon [8]\to [12]$
 such that $R(x_{\sigma(1)},\dotsc,x_{\sigma(8)})$ holds. 
Since $R$ is constant on $V_{\sigma(1)}\times\dotsb\times V_{\sigma(8)}$, 
we obtain at least as many space crossings as the number of
quadruples of edges of the form $e_1,e_2,e_3,e_4$, where $e_i$
is between $V_{2i-1}$ and $V_{2i}$.
\end{proof}

\section*{Straight-line spatial drawing with very few space crossings}\label{pf_thm_constr}
\paragraph{Review of stair-convexity}
To prove Theorem~\ref{thm_constr} we employ stair-convexity, which is
a method to make constructions in $\R^d$ in such a way that
convex sets, which are geometric objects, are replaced by their combinatorial
cousins, stair-convex sets. 

The basis for the connection between convexity and stair-convexity is the 
\emph{stretched grid} $G_s=G_s(n)$ which is the Cartesian product 
$X_1\times X_2\times \dotsb\times X_d$, where $X_1,X_2,\dotsc,X_d$ 
are ``fast-growing'' sequences, with each $X_i$ growing much faster than $X_{i-1}$.
Let $X_i=\{x_{i1},\dotsc,x_{in}\}$. The actual choice of $X_1,X_2,\dotsc,X_d$ is 
not important, as long as they grow quickly enough.

More precisely, for each coordinate $i=1,\dotsc,d$ there is a relation $\prec_i$,
such that the condition on the growth of $X_i$ is that $1=x_{i1}\prec_i x_{i2}\prec_i \dotsb \prec_i x_{im}$.
The relation $\prec_i$ is not a linear relation, but it is transitive, and is compatible
with the usual linear ordering on $\R$ in the sense that $A\prec_i B$ implies $A<B$.

Since the coordinates in $G_s$ grow very fast, to visualize and to work with the grid  
it is convenient to rescale $G_s$. Let $BB(G_s)=[1,x_{1m}]\times\dotsb\times[1,x_{dm}]$ be the 
``bounding box'' of $G_s$. Let the \emph{uniform grid} be
\begin{equation*}
G_u = G_u(n) \defeq \bigl\{0,\frac{1}{n-1},\frac{2}{n-1},\dotsc,\frac{n-1}{n-1}\bigr\}^d,
\end{equation*}
and pick a bijection $\pi\colon BB(G_s)\to [0,1]^d$ that maps $G_s$ onto $G_u$
and preserves ordering in each coordinate.


\begin{figure}\centering
\parbox{2in}{\begin{tikzpicture}[scale=0.35]
\def\gridsize{14}
\foreach \x in {0,...,\gridsize}
  \foreach \y in {0,...,\gridsize}
   {
    \ifthenelse{\(\equal{\x}{3}\AND \equal{\y}{3}\)\OR\(\equal{\x}{8}\AND \equal{\y}{10}\)\OR\(\equal{\x}{5}\AND \equal{\y}{12}\)\OR\(\equal{\x}{11}\AND \equal{\y}{4}\)}
    {
     \fill (\x,\y) circle (4.2pt) +(-45:0.48) node {$\scriptstyle\ifthenelse{\equal{\x}{3}}{a}{\ifthenelse{\equal{\x}{8}}{b}{\ifthenelse{\equal{\x}{5}}{c}{d}}}$};
    }
    {
     \fill (\x,\y) circle (2.3pt);
    }
   }
\input{plot1.dat};
\input{plot2.dat};
\end{tikzpicture}
\caption{Image under $\pi$ of line segments $[a,b]$ and $[c,d]$}%
\label{fig_normal}}\quad\qquad%
\parbox{2in}{
\begin{tikzpicture}[scale=0.35]
\def\gridsize{14}
\foreach \x in {0,...,\gridsize}
  \foreach \y in {0,...,\gridsize}
   {
    \ifthenelse{\(\equal{\x}{3}\AND \equal{\y}{3}\)\OR\(\equal{\x}{8}\AND \equal{\y}{10}\)\OR\(\equal{\x}{5}\AND \equal{\y}{12}\)\OR\(\equal{\x}{11}\AND \equal{\y}{4}\)}
    {
     \fill (\x,\y) circle (4.2pt) +(-45:0.48) node {$\scriptstyle\ifthenelse{\equal{\x}{3}}{a}{\ifthenelse{\equal{\x}{8}}{b}{\ifthenelse{\equal{\x}{5}}{c}{d}}}$};
    }
    {
     \fill (\x,\y) circle (2.3pt);
    }
   }
\draw plot coordinates{(3.0,3.0) (3.0,10.0) (8.0,10.0)};
\draw plot coordinates{(5.0,12.0) (11.0,12.0) (11.0,4.0)};
\end{tikzpicture}
\caption{Image under $\pi$ of stair-paths $\sigma(a,b)$ and $\sigma(c,d)$}%
\label{fig_stair}}
\end{figure}

The figure \ref{fig_normal} above shows the image under $\pi$ of two straight line 
segments connecting the grid points for $d=2$. As the 
uniform grid becomes finer, the straight line segments become closer to a piecewise
linear curve, the stair-path. A \emph{stair-path} joining points 
$a=(a_1,a_2,\dotsc,a_d)$ and $b=(b_1,b_2,\dotsc,b_d)$ consists of at most $d$ closed
line segments, each parallel to a different coordinate axis. The 
definition goes by induction on $d$. For $d=1$, $\sigma(a,b)$ is 
simply the segment $ab$. For $d\geq 2$, after possibly interchanging 
$a$ and $b$, assume $a_d\leq b_d$. We set $a'=(a_1,a_2,\dotsc,a_{d-1},b_d)$
and let $\sigma(a,b)$ be the union of the segment $aa'$ and of the stair-path 
$\sigma(a',b)$, which is defined recursively after ``forgetting'' the 
(common) last coordinate of $a'$ and $b$. A set $S\subseteq \R^d$ is 
\emph{stair-convex} if for every $a,b\in S$ 
we have $\sigma(a,b)\subseteq S$. Since the intersection of stair-convex 
sets is stair-convex, we can define \emph{stairconvex hull} 
of a set $S\subset \R^d$ as the intersection of all stair-convex 
sets containing $S$.

Two points $(a_1,a_2,\dotsc,a_d)$ and $(b_1,b_2,\dotsc,b_d)$ in $BB(G_s)$ 
are \emph{$k$-far apart in $i$'th coordinate}, if there are $k-1$ real 
numbers $r_1,\dotsc,r_{k-1}$ such that either 
$a_i\prec_i r_1\prec \dotsb r_{k-1}\prec_i b_i$ or 
$b_i\prec_i r_1\prec \dotsb r_{k-1}\prec_i a_i$. 
Otherwise, we say that $a$ and $b$ are $k$-close in 
$i$'th coordinate. If $a$ and $b$ are $k$-close in every 
coordinate, then we say that $a$ and $b$ are 
\emph{$k$-close}. If $a$ and $b$ lie on $G_s$,
then they are $k$-close if $\pi(a)$ and $\pi(b)$, 
which are points of $G_u$, are separated by fewer
than $k$ points in each coordinate. In the picture above,
the points $a$ and $b$ are $6$-close, but not $5$-close.
For $a,b\in BB(G_s)$ put $\dist(a,b)$ be the least
integer $k$ such that $a$ and $b$ are $k$-close. Note that
$\dist$ satisfies the triangle inequality 
$\dist(a,b)\leq\dist(a,c)+\dist(b,c)$.


Several results capture the intuition that the image of a convex 
set in $BB(G_s)$ looks like a stair-convex set. The following lemma 
of Nivasch is the form that we need. 
\begin{lemma}[Lemma~2.11 in \cite{nivasch_thesis}]
Let $a,b$ be two points in $BB(G_s)$, and let $ab$ and $\sigma(a,b)$ 
be the line segment and the stair-path between $a$ and $b$, respectively. 
Then every point in $ab$ is $1$-close to some point of $\sigma(a,b)$ and vice versa.
\end{lemma}
\begin{corollary}\label{cor_nivasch}
Suppose $a,b,a',b'$ are points in $BB(G_s)$. If the segments $ab$ and 
$a'b'$ intersect, then there are points $c\in \sigma(a,b)$ and 
$c'\in\sigma(a',b')$ that are $2$-close.
\end{corollary}
\begin{proof} Let $c$ and $c'$ be $1$-close to $ab\cap a'b'$. Then $\dist(c,c')\leq 2$ holds by the triangle inequality. \end{proof}

\paragraph{Proof of Theorem~\ref{thm_constr}}
We shall now describe a straight-line drawing with few space crossings.
From now on we fix $d=3$ and pick a particular choice of stretched grid $G_s=G_s(5n)$
with $(5n)^3$ points. We shall also work with the subgrid $G_s'\subset G_s$
that consists of the points of the form $(x_{1i_1},x_{2i_2},x_{3i_3})$
with $5\mid i_1,i_2,i_3$. The subgrid $G_s'$ has $n^3$ points.
Let $p(i)=(x_{1i},x_{2i},x_{3i})$ be the $i$'th point on the ``diagonal'' of~$G_s$. 

Let $G$ be the graph on the vertex set $\{1,2,3,\dotsc,n\}$ 
with $i$ and $j$ forming an edge if $\abs{i-j}\leq D$,
where $D=2m/n$. The \emph{standard drawing of $G$} is 
one in which the vertex $i\in V(G)$ is represented by the point 
$p(5i)$, and all the edges are straight-line segments. 
Note that in this drawing all the vertices lie on the subgrid $G_s'$,
and thus no pair of them is $5$-close. Moreover, if the stair-paths
$\sigma(a_1,b_1)$ and $\sigma(a_2,b_2)$ with $a_1,a_2,b_1,b_2\in G_s'$
do not intersect, then no point of $\sigma(a_1,b_1)$ is $5$-close to
a point of $\sigma(a_2,b_2)$.

We say that four stair-paths form a \emph{space stair-crossing} if there 
is another stair-path that meets all four stair-paths. The \emph{standard
stair-drawing of $G$} is one in which vertex $i\in V(G)$ is represented
by the point $p(5i)$, and all edges are stair-paths.

The following two  lemmas imply Theorem~\ref{thm_constr}.
\begin{lemma}\label{lem_stair_approx}
Let $s_1,t_1,\dotsc,s_4,t_4\in V(G)$ be eight distinct
vertices of $G$. Then the edges $s_1t_1,\dotsc,s_4t_4$
form a space crossing in the standard drawing of $G$
only if they form a space stair-crossing
in the standard stair-drawing of $G$.
\end{lemma}
\begin{lemma}\label{lem_stair_char}
Let $s_1,t_1,\dotsc,s_4,t_4\in [0,1]$ be distinct
vertices of $G$. 
Let $I_i=[s_i,t_i]$ be the interval spanned by $s_i$ and $t_i$.
Then the four vertex-disjoint edges $s_1t_1,\dotsc,s_4t_4$ 
form a stair-crossing in the standard stair-drawing of 
$G$ only if for each $i=1,\dotsc,4$ there is
at least one $j\neq i$ so that $I_i\cap I_j\neq \emptyset$
\end{lemma}
\begin{proof}[Proof that Lemmas~\ref{lem_stair_approx} and
\ref{lem_stair_char} imply Theorem~\ref{thm_constr}]
The graph $G$ has $n$ vertices and 
more than $D n/2=m$ edges. 
The four vertex-disjoint edges $s_1t_1,\dotsc,s_4t_4\in E(G)$ form
a space crossing only if the union of the four 
intervals $[s_1,t_1],\dotsc,[s_4,t_4]$
has at most two connected components.

There are $\frac{1}{4!}\binom{2}{2}\binom{4}{2}\binom{6}{2}\binom{8}{2}=105$
order types of four unlabeled endpoint-disjoint intervals (each order
type corresponds to a perfect matching on $8$ labeled points).
Each order type that consists of  
$r\leq 2$ connected component gives rise to at most $n^r D^{8-r}$ space
crossings in the standard drawing of $G$. Indeed, there are $n^r$
ways to choose the leftmost points of the intervals, and once
those are specified, it only suffices to specify the distances between
the consecutive points in a connected component, and these
distances are bounded by $D$. Thus, there are at most 
$105n^2 D^6=6720 m^6/n^4$ space crossing in the standard drawing of $G$. 
\end{proof}

\begin{proof}[Proof of Lemma~\ref{lem_stair_approx}]
Suppose $s_1t_1,\dotsc,s_4t_4$ are edges of $G$ 
forming a space crossing, and let $l$ be the line that 
meets these four edges. Let $r_1$ and $r_2$ be the 
intersection points of $l$ with $BB(G_s')$. Then by 
Corollary~\ref{cor_nivasch} the stair-path $\sigma(r_1,r_2)$
is $2$-close to the stair-paths
$\sigma(s_1,t_1),\dotsc,\sigma(s_4,t_4)$.

Let $r_1'$ and $r_2'$ be the points of $G_s'$ such that
$\dist(r_1,r_1')\leq 3$ and $\dist(r_2,r_2')\leq 3$. 
Since $\sigma(r_1,r_2)$ is $2$-close to $\sigma(s_1,t_1)$,
by the triangle inequality, $\sigma(r_1',r_2')$ is $5$-close
to $\sigma(s_1,t_1)$. Since $r_1',r_2',s_1,t_1$ belong to
$G_s'$, that means that $\sigma(r_1',r_2')$ and $\sigma(s_1,t_1)$
intersect. Similarly, $\sigma(r_1',r_2')$ intersects $\sigma(s_i,t_i)$
for $i=1,\dotsc,4$, and the edges $s_1t_1,\dotsc,s_4t_4$ 
form a stair-crossing.
\end{proof}

\begin{proof}[Proof of Lemma~\ref{lem_stair_char}]
Every stair-path is a subset of one of the three 
types of sets 
\begin{enumerate}
\item $L_1(x_0,y_0,y_1,z_1)\defeq \sigma((x_0,y_0,0),(+\infty,y_1,z_1))$,
\item $L_2(x_0,y_0,y_1,z_1)\defeq \sigma((x_0,y_0,0),(-\infty,y_1,z_1))$,
\item $L_3(x_0,y_0,x_1,z_1)\defeq \sigma((x_0,y_0,0),(x_1,-\infty,z_1))$.
\end{enumerate}
We call $L_1,L_2,L_3$ \emph{stair-lines}. The numbers 
$x_0,y_0,\dotsc$ are the \emph{coordinates} of stair-lines.

Let $L$ be a line that meets the four edges $s_1t_1,\dotsc,s_4t_4$.
We say that $L$ shares the coordinate with the edge $s_it_i$ if that coordinates 
is equal to either $s_i$ or $t_i$. Note that
since the edges are vertex-disjoint, no coordinate
of $L$ can be shared with two distinct edges.
Since the edge $s_it_i$ is represented by the stair-path
$\sigma((s_i,s_i,s_i),(t_i,t_i,t_i))$, it
meets $L$ only if it shares a coordinate with $L$.
It follows that each of the four coordinates of 
$L$ is shared with a unique edge.

Suppose the edge $s_it_i$ shares the coordinate $c$ with $L$.
Let $p_i$ be an intersection point of $L$ with $s_it_i$.
The point $p_i$ shares at least two coordinates with $L$,
one of which is $c$. The other coordinate $c'$ is between 
$s_i$ and $t_i$. Thus if $s_jt_j$ 
is the edge that shares the coordinate $c'$ with $L$,
then the intervals $[s_i,t_i]$ and $[s_j,t_j]$ intersect.
\end{proof}

\section*{The proof of Lemma~\ref{lem_sametype}}\label{pf_lem_sametype}
To simplify the proof we will need 
to work in slightly greater
generality than subsets of a fixed Euclidean space.
First, we permit $\F_1,\F_2,\dotsc$ to be multisets (this will
come handy in the proof of Theorem~\ref{semialgst}),
and we permit different multisets to belong to different spaces.
To keep track of these spaces, we introduce 
a bit of notation. For a $k$-tuple $\bd=(d_1,\dotsc,d_k)\in \N^k$ 
we define $\R^{\bd}\defeq\R^{d_1}\times\dotsb\times \R^{d_k}$. 
For simplicity of notation we adopt the 
convention that $\R^{d_i}$ denotes the $i$'th component
of $\R^{\bd}$ even if there are several $i$ that share
the same value of $d_i$. 
The \emph{number of terms} of a polynomial $f$ on $i$'th component,
denoted $t_i(f)$,
is the number of monomials of $f$ when treated as a polynomial
on $\R^{d_i}$ (with other coordinates treated as fixed). 
For example, if $\bd=(1,1)$ and $f\colon \R^\bd\to \R$ is defined by
$f(x_1,x_2)=4+2x_1x_2+3x_1x_2^2+x_1x_2^3+7x_1^2x_2$, then
$t_1(f)=3$, and $t_2(f)=4$, though $f$ has five terms.
The sign of a number $x\in \R$ is $+1$ if $x>0$, $-1$ if $x<0$
and $0$ if $x=0$.

By Tarski--Seidenberg theorem (see \cite[Theorem~2.77]{basu_pollack_roy}) 
every semialgebraic relation is
equivalent to a quantifier-free semi-algebraic relation. Thus the
following result implies Lemma~\ref{lem_sametype}.

\begin{theorem}\label{semialgst}
Let $f_1,\dotsc,f_J\colon \R^{\bd}\to \R$ be a family of polynomials.
Suppose $\F_i\subset \R^{d_i}$ for $i=1,\dotsc,k$ are
finite multisets of points. Then there are submultisets
$\F_i'\subset \F_i$ such that
\begin{enumerate}
\item $\F_i'$ are large: $\abs{\F_i'}\geq \veps \abs{\F_i}$,
with $\veps=\prod_{j=1}^J 3^{-3^{t_2(f_j)+\dotsb+t_k(f_j)}}$;
\item For each $j$ the sign of $f_j(p_1,\dotsc,p_k)$ is same for all
the choices of $p_i\in \F_i'$.
\end{enumerate}
\end{theorem}
(The fact that the expression for $\veps$ does not depend on 
$t_1(f_j)$ is not a typo, but an artifact of the proof.)

We use a version of Yao--Yao lemma 
\cite{yao_yao} due to Lehec \cite{lehec}. 
Recall that a \emph{convex cone} of vectors
$v_1,\dotsc,v_r$ is the set of all their non-negative linear
combinations, $\convcone(v_1,\dotsc,v_r)\defeq\sum a_i v_i$ with $a_i\geq 0$.
\begin{lemma}[Theorem~3 and Proposition~4 in \cite{lehec}]\label{lem_yao_yao}
Let $\mu$ be a probability Borel measure on $\R^d$ such that $\mu(H)=0$
on every affine hyperplane $H$. Then there is a way choose
the origin of the coordinates in $\R^d$ and $2^d$ convex cones such that:
\begin{enumerate}
\item The union of the cones is $\R^d$, and the cones are disjoint
apart from the boundaries;
\item Each cone has measure $1/2^d$ with respect to $\mu$;
\item Every closed halfspace that contains the origin also contains
one of the cones;
\item Each cone is a convex cone of only $d$ vectors.
\end{enumerate}
\end{lemma}
By the standard approximation argument it follows that if $\F$ is any 
finite multiset of points
in $\R^d$, then there is a partition of $\R^d$ as in the lemma above such that
each cone contains at least $\abs{\F}/2^d$ points of $\F$. Furthermore,
if $\convcone(\{v_1,\dotsc,v_d\})$ is one of the convex cones, then
for large enough $t>0$ we have 
$\convcone(\{v_1,\dotsc,v_d\})\cap P=\conv(\{0,tv_1,\dotsc,tv_d\})\cap P$.
We thus obtain
\begin{corollary}\label{cor_yao_yao}
Suppose $\F$ is a finite multiset in $\R^d$. Then there is a 
point $p$ and $2^d$ $d$-dimensional closed simplices 
$\Delta_1,\dotsc,\Delta_{2^d}\subset \R^d$ such that
\begin{enumerate}
\item The interiors of the simplices $\Delta_1,\dotsc,\Delta_{2^d}$ are disjoint;
\item For each $j=1,\dotsc,2^d$ the number of points in $j$'th simplex is 
$\abs{\F\cap \Delta_j}\geq \abs{\F}/2^d$;
\item The point $p$ is a vertex of each $\Delta_j$ for $j=1,\dotsc,2^d$;
\item Every closed halfspace that contains $p$ also contains one of $\Delta_j$'s.
\end{enumerate}
\end{corollary}

The following lemma is a minor variation on the standard linearization
argument (see \cite{agarwal_matousek} for example).
\begin{lemma}\label{lem_linearone}
Let $R$ be a commutative ring. Let $f\colon R^d\to R$ be a polynomial with $t$ non-constant 
terms. Then there is a map $\pi\colon R^{d}\to R^{t}$ and a linear polynomial
$f'\colon R^{t} \to R$ such that $f=f'\circ \pi$.
\end{lemma}
\begin{proof}
Let $1=g_0,g_1,\dotsc,g_{t}$ be the set of all monomials appearing in
$f$. Let $f=\sum \alpha_i g_i$. 
Define $\pi(x)=(g_1(x),\dotsc,g_{t}(x))$, and 
$f'(z_0,z_1,\dotsc,z_{t})=\sum \alpha_i z_i$. The identity 
$f=f'\circ \pi$ is clear.
\end{proof}

\begin{proof}[Proof of Theorem~\ref{semialgst}]
The proof is by induction on $k$. The base case $k=1$ is trivial
because for at least one third of all the points $x_1\in\F_1$ the 
sign of $f_j(x_1)$ is the same.
Suppose $k\geq 2$, and the theorem is known to hold for $k-1$.
It suffices to prove the result for a single polynomial, which 
we shall call $f$. 
Think of $f$ as a polynomial with $t_k(f)$ terms 
on $\R^{d_k}$.  By Lemma~\ref{lem_linearone} there is 
map $\pi\colon \R^{d_k}\to \R^{t_k(f)}$ and a polynomial 
$f'\colon \R^{\bd'}\to \R$, which is linear on $\R^{t_k(f)}$,
such that 
\begin{equation*}
f(x_1,x_2,\dotsc,x_{k-1},x_k)=f'(x_1,x_2,\dotsc,x_{k-1},\pi(x_k)).
\end{equation*}

Apply Corollary~\ref{cor_yao_yao} to the multiset
$\pi(\F_k)$.
Let $\Delta_1,\dotsc,\Delta_{2^{d_k'}}\subset \R^{t_k(f)}$ be the 
simplices whose existence the corollary guarantees. They have a
total of at most $1+t_k(f') 2^{t_k(f)}\leq 3^{t_k(f)}$ 
vertices, which we denote by $v_1,\dotsc,v_M$ where $M\leq 3^{t_k(f)}$.
Each of the simplices contains at least%
\footnote{It is here that we use that $\F_k$ is a multiset. If $\F_k$
was defined to be a set, then $\pi(\F_k)$ might have 
had fewer elements than $\F_k$.} 
$\abs{\F_k}/2^{t_k(f)}$ points of $\pi(\F_k)$.

Since the polynomial $f'$ is linear in $x_k$ each choice of 
$x_i\in \R^{d_i}$ ($i=1,\dotsc,k-1$) gives the hyperplane in $\R^{t_k(f)}$,
namely the hyperplane $H(x_1,\dotsc,x_{k-1})=\{x\in\R^{d_k'} : 
f'(x_1,\dotsc,x_{k-1},x)=0\}$. Define $H^+(x_1,\dotsc,x_{k-1})$
and $H^-(x_1,\dotsc,x_{k-1})$ as the two closed halfspaces 
that $H(x_1,\dotsc,x_{k-1})$ bounds in the obvious way. 
By Corollary~\ref{cor_yao_yao}
either $H^-$ or $H^+$ contains some $\Delta_j$.

For each point $v_m$ define the polynomial $g_m$ by
$g_m(x_1,\dots,x_{k-1})=f'(x_1,\dotsc,x_{k-1},v_m)$.
Note that the indices $j$ for which $\Delta_j$ is contained in 
$H^+(x_1,\dotsc,x_{k-1})$ depends only on the signs 
of $g_m(x_1,\dotsc,x_{k-1})$, and
similarly for the indices $j$ for which $\Delta_j\subset 
H^-(x_1,\dotsc,x_{k-1})$.

Since $t_i(g_m)\leq t_i(f)$ for each $i=1,\dotsc,k-1$,
by the induction hypothesis there are subsets 
$\F_i'\subset \F_i$ for $i=1,\dotsc,k-1$ 
of size $\abs{\F_i'}\geq 3^{-M\cdot 3^{t_2(f)+\dotsb+t_{k-1}(f)}} \abs{\F_i}$
such that for all $m=1,\dotsc,M$ the sign of 
$g_m(x_1,\dotsc,x_{k-1})$ does not depend on the choice of $x_i\in \F_i'$.
Denote this sign by $\epsilon_m$. 
Therefore there is a single
$j$ and a non-zero sign $s$ such that $\Delta_j$ 
is contained in $H^s(x_1,\dotsc,x_{k-1})$
for each choice $x_i\in \F_i'$. Without loss of generality
$s=+1$, which means that $\epsilon_m\geq 0$ for every
vertex $v_m$ of $\Delta_j$. Let $\sigma$ be the face of 
$\Delta$ spanned by the vertices $v_m$ for which $\epsilon_m=0$.
Thus 
$f'(x_1,\dotsc,x_{k-1},x_k)=0$ ($x_i\in \F_i'$) if $x_k\in \sigma$ and 
$f'(x_1,\dotsc,x_{k-1},x_k)>0$ if $x_k\in \Delta_j\setminus \sigma$.
One of the two alternatives holds for at least half of the points
in $\pi(\F_k)\cap \Delta_j$, and since $2^{-t_k(f)-1}\geq 3^{-3^{t_k(f)}}$
the theorem follows.
\end{proof}

\section*{Higher dimensions and other questions}
Throughout the paper we spoke of ``the'' space crossing number, but
it is but one of a family of similar quantities. Similarly to the way
the crossing number measures the complexity of planar embeddings, these 
quantities measure the complexity of embeddings into higher-dimensional 
Euclidean spaces. We give some examples: 
\begin{enumerate}
\item For an embedding of graph $G\to \R^3$ one can count not the quadruples, but 
the triples of edges crossed by a line. The methods in this paper
easily adapt to show that the corresponding crossing number $\crn_3(G)$
satisfies $\crn_3(G)\geq c(\abs{E}^4/\abs{V}^2\log^2\abs{V})$, and there
are graphs $G$ such that $\rcrn_3(G)\leq C(\abs{E}^4/\abs{V}^2)$.
\item For an embedding of a graph $G\to \R^4$ there are at least two 
kinds of objects to consider: the lines that pierce three edges of $G$, 
and $2$-planes that pierce $6$ edges of $G$. The simple dimension-counting
shows that for generic embeddings, there are finitely many such lines and 
$2$-planes.
\item More generally, one can count the number of $(d-2)$-dimensional
planes through $2(d-1)$ edges of $G\to \R^d$. The case $d=2$ is the classical
crossing number, whereas $d=3$ is the space crossing number of the present paper.
Theorem 3 from \cite{karasev_highdim} can be used to give lower bounds
on these higher crossing numbers, but we lack the constructions for the upper bounds.
 
\item One can consider representations of a $3$-uniform hypergraph in $\R^4$ 
by means of (topological) triangles, and count the number of triples
of triangles that meet at a single point. However, it is an open problem even
to shown that in every $3$-uniform hypergraph with more than $Cn^2$ edges there is a 
single pair of intersecting triangles!
\end{enumerate}

There are several more question about the crossing numbers we defined:
\begin{enumerate}
\item 
A
result about crossing number with many applications is the bisection
width inequality (proved independently in \cite[Theorem~2.1]{pach_shahrokhi_szegedy_bisect}, 
extending the proof in \cite{orig_crossing_leighton} for the bounded-degree graphs). 
The bisection width inequality states that
$\crn(G)^2\geq c_1 b^2(G)-c_2\sum_{v\in V(G)} \deg_G(v)^2$, where $b(G)$ is the bisection width
of the graph $G$. Is there an analogous inequality for the
space crossing number that is of comparable strength to Theorem~\ref{thm_expand}?
\item Is the family of graphs with $\crn_4(G)=0$ a minor-closed family? 
\item Is it true that $\crn_4(G)=0$ if and only if $\rcrn_4(G)=0$?
\end{enumerate}
We guess that the answers are 1) yes, 2) no, and 3) no.

\paragraph{Acknowledgments.} We thank Andr\'as Juh\'asz for discussions about Lemma~\ref{lem_linking}.
We also grateful to him and Saul Schleimer for helping us to find \cite{viro}.  We thank J\'anos Pach
and Boris Aronov for helpful conversations. The work was partly done with support
of Discrete and Computational Program in Bernoulli Centre, EFPL, Lausanne. The second
author acknowledges the support of Conacyt.

\bibliographystyle{alpha}
\bibliography{spacecrossing}


\end{document}